\documentclass[3p]{elsarticle}
\usepackage[utf8]{inputenc}
\usepackage{lmodern}

\usepackage{amsthm}
\usepackage{todonotes}
\newtheorem{theorem}{Theorem}
\newtheorem{lemma}[theorem]{Lemma}
\newtheorem{remark}[theorem]{Remark}

\usepackage{amsmath}
\usepackage{amssymb}
\usepackage{soul}

\usepackage{pgfplots}

\usepackage{tikz}
\usetikzlibrary{spy,shapes.misc,patterns,snakes,calc}

\usepackage{booktabs}
\usepackage[scientific-notation=true]{siunitx}
\usepackage{subcaption}
\newcommand{\foralls}{\forall\,}

\newcommand{\R}{\mathbb{R}}

\newcommand{\enorm}[1]{\vert\!\vert\!\vert #1 \vert\!\vert\!\vert}

\newcommand{\pO}{\partial \Omega}

\newcommand{\dx}{\,\mathrm{d}x}

\newcommand{\ds}{\,\mathrm{d}s}

\newcommand{\mT}{\mathcal{T}}

\newcommand{\dofs}{d.o.f.'s}
\newcommand{\dof}{d.o.f.}

\newcommand{\comment}[1]{}

\biboptions{sort&compress}

\journal{Computers \& Mathematics with Applications}

\begin{document}

\begin{frontmatter}

\title{A note on the penalty parameter in Nitsche's method for unfitted boundary value problems}

\author[TUE]{Frits de Prenter\corref{mycorrespondingauthor}}
\cortext[mycorrespondingauthor]{Corresponding author. Tel.: +31 61 516 2599}
\ead{f.d.prenter@tue.nl}

\author[GAU]{Christoph Lehrenfeld}
\ead{lehrenfeld@math.uni-goettingen.de}

\author[UMU]{Andr\'{e} Massing}
\ead{andre.massing@umu.se}

\address[TUE]{Department of Mechanical Engineering, Eindhoven University of Technology, The Netherlands}
\address[GAU]{Institute for Numerical and Applied Mathematics, University of G\"{o}ttingen, Germany}
\address[UMU]{Department of Mathematics and Mathematical Statistics, Ume\aa\ University, Sweden}

\begin{abstract}
  Nitsche's method is a popular approach to implement Dirichlet-type
  boundary conditions in situations where a strong imposition is
  either inconvenient or simply not feasible. The method is widely
  applied in the context of unfitted finite element methods. From the
  classical (symmetric) Nitsche's method it is well-known that the
  stabilization parameter in the method has to be chosen sufficiently
  large to obtain unique solvability of discrete systems. In this
  short note we discuss an often used strategy to set the 
  stabilization parameter and describe a possible problem that can
  arise from this. We show that in specific situations error bounds
  can deteriorate and give examples of computations where Nitsche's
  method yields large and even diverging discretization errors.
\end{abstract}

\begin{keyword}
 Nitsche's method \sep unfitted/immersed finite element methods \sep penalty/stabilization parameter \sep accuracy \sep stability \sep error analysis
\end{keyword}

\end{frontmatter}

\section{Introduction}

We consider the discretization of an unfitted Poisson problem. The
problem domain is described separately from the encapsulating mesh, on which a finite
element basis is defined. We consider the restriction of this finite
element space with respect to the problem domain. Such an approach is
used in many methods which are similar in virtue, 
\emph{e.g.,} 
the fictitious domain method
\cite{glowinskietal94}, the cut finite element
method (CutFEM) \cite{burman2012fictitious,burman2014cutfem,MassingLarsonLoggEtAl2013a,BurmanClausMassing2015}, 
the finite cell method (FCM)
\cite{Parvizian2007,Duester2008,Schillinger2012,Rank2012,Ruess2013,Ruess2014,Schillinger2015},
immersogeometric analysis \cite{Kamensky2015,Varduhn2016,Xu2016},
the immersed boundary method \cite{Peskin2002}, the unfitted
discontinuous Galerkin method (UDG)
\cite{bastian2009unfitted,massjung12}, the extended finite element
method (XFEM) \cite{belytschkoetal01,fries2010extended,
hansbo2002unfitted, gross04, Becker20093352, BeckerBurmanHansbo2010},
and several others.  To impose essential boundary conditions, many of
these methods apply some version of the classical Nitsche's method
\cite{Nitsche1971}. 

This requires stabilization by a penalization parameter to preserve the coercivity of the bilinear operator.
Without additional stabilization of cut elements with \emph{e.g.,} ghost penalty terms \cite{Burman2010,burman2012fictitious}
-- which is customary for methods referred to as CutFEM --
this penalty parameter depends on the shape and size of the cut elements,
\emph{i.e.,} it depends on the position of the geometry relative to the
computational mesh. A typical choice which is sufficient to provide
unique solvability of discrete problems, is to choose the stabilization
parameter as an element wise constant and proportional to the ratio between
the surface measure of the intersection between an element and the
boundary and the volume measure of the intersection of the same
element and the domain. 
As this ratio can become arbitrarily large,
the stabilization parameter is not generally bounded.  

In the mathematical literature, \emph{e.g.,} \cite{burman2012fictitious,burman2014cutfem}, unfitted Nitsche formulations are generally supplemented by the aforementioned additional stabilization to bound the stabilization parameter in order to prove properties of the method.
In our opinion, the importance of this has not sufficiently been addressed in the literature however.
Moreover, the method has effectively (and successfully) been applied without additional stabilization in the engineering literature, see \emph{e.g.,} \cite{Ruess2013,Ruess2014,Schillinger2015}.
With this note, we therefore detailedly treat the possible problem that can occur when no additional measures are taken to bound the stabilization parameter and aim to provide a disclaimer for directly applying the classical form of Nitsche's method to immersed problems.
To this end, we discuss and analyze the method and demonstrate that it can lead to poor results in the discretization error when the geometry intersects the computational domain such that large values of the stabilization parameter are required.
We also investigate different situations where degenerated geometry configurations yield satisfactory and unsatisfactory results and provide a possible interpretation.
We further review alternative formulations and possible modifications and stabilizations of the method.

Section~\ref{sec:prob} of this note introduces a model problem, followed
by the Nitsche's method under consideration in
Section~\ref{sec:method}. In Section~\ref{sec:analysis} we carry out a
simple error analysis of the method in a norm which is natural to the
formulation. This analysis reveals the possible large discretization
errors for unfortunate cut configurations. In Section~\ref{sec:numex}
we give examples of bad cut configurations and show how these can
lead to large and even degenerating discretization errors. Finally, we list alternative
approaches to impose boundary conditions and variants of the classical
Nitsche's method to circumvent this possible problem in Section~\ref{sec:alternatives}.

\section{The model problem} 
\label{sec:prob} 
As a model for more general elliptic boundary value problems, we
consider the Poisson problem with inhomogeneous Dirichlet boundary
data posed on an open and bounded domain $\Omega \subset \R^d$ with
Lipschitz boundary $\pO$:
\begin{subequations} \label{eq:ellmodel}
  \begin{align}
    - \Delta u &= \, f \quad \text{in}~~ \Omega, \\
    u &= \, g \quad \text{on}~~ \pO.
  \end{align}
\end{subequations}
We assume that the properties of the domain $\Omega$
imply that a shift theorem of the form:
\begin{align}
  \Vert u \Vert_{H^2(\Omega)}
  \lesssim
  \Vert f \Vert_{L^2(\Omega)} + \Vert g \Vert_{H^{\frac12}(\pO)},
\end{align}
holds whenever $f\in L^2(\Omega)$ and $g \in
H^{\frac12}(\pO)$, \emph{e.g.,} when $\pO$ is smooth or $\Omega \subset \R^2$
is a convex domain with piecewise $C^2$ boundary~\cite{Grisvard1985}.
In this note $\lesssim$ is used to denote an inequality with a constant that is independent of the mesh.
This problem has the standard well-posed weak
formulation: Find $u \in H^1_g(\Omega)$, such that:
\begin{equation}
  \int_{\Omega} \nabla u \nabla v \, dx = \int_{\Omega} f v \, dx \quad \mathrm{for\ all\ } v \in H^1_0(\Omega).
\end{equation}
In this formulation $H^1_0(\Omega)$ and $H^1_g(\Omega)$ are the
standard Sobolev spaces with corresponding homogeneous and
inhomogeneous boundary data. For simplicity, we assume that Dirichlet
boundary conditions are posed on the whole boundary of $\Omega$. 
However, the results presented in this work extend to the case where
Dirichlet boundary data is only prescribed on a part of the boundary
as demonstrated in Section~\ref{sec:numex}.

\section{A Nitsche-based unfitted finite element method} 
\label{sec:method} 
We consider a (triangular or
rectilinear) shape regular \emph{background} discretization $\widetilde{\mT}$
which encapsulates the domain $\Omega$.
To each element $T \in \widetilde{\mT}$ we associate the local mesh size 
$h_T={\mathrm{diam}}(T)$ and define
the global mesh size by $h = \max_{T\in\widetilde{\mT}} h_T$.
We define the \emph{active} mesh $\mT$ 
by:
\begin{align} 
  \mT &= \{ T \in \widetilde{\mT} : T \cap \Omega \neq \emptyset \},
  \label{eq:narrow-band-mesh}
\end{align}
and set $\widetilde{\Omega} = \bigcup_{T\in \mT} T$, denoting the union of all
elements from the active mesh.
The set of geometric entities are illustrated in
Figure~\ref{fig:domain-set-up}.
\begin{figure}[htb]
  \begin{center}
    \includegraphics[width=0.5\textwidth]{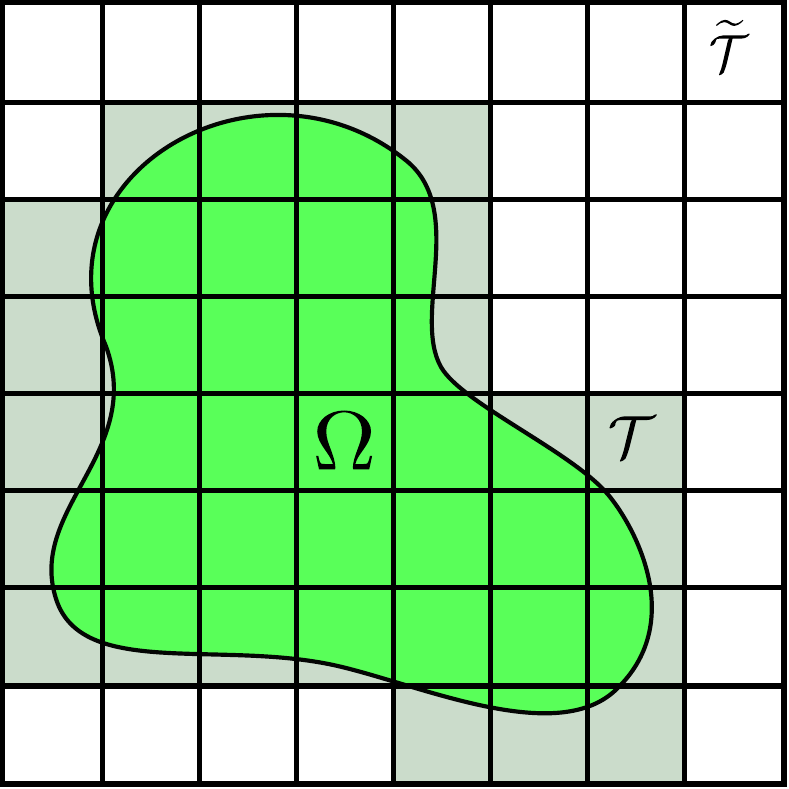}
    \caption{Example of a domain $\Omega$ with background mesh $\widetilde{\mT}$ and active mesh $\mT$.}
    \label{fig:domain-set-up}
  \end{center}
\end{figure}
On the active mesh $\mT$ we define the finite dimensional function space:
\begin{equation}
  V_h = \{ v \in C(\widetilde{\Omega}): v|_T \in  P_k(T), T \in \mT \},
\label{eq:Vh-def}
\end{equation}
consisting of continuous piecewise polynomials
of order $k$.
A simple prototype formulation for a Nitsche-based unfitted finite element
method is to find
$u_h \in V_h$, such that:
\begin{equation}
  \label{eq:formulation}
  a_h(u_h, v_h) = l_h(v_h) \quad \foralls v_h \in V_h,
\end{equation}
  where:
  \begin{align}
    a_h(u_h,v_h) & := \int_{\Omega} \nabla u_h \nabla v_h \dx 
      + \int_{\pO} (-\partial_n u_h) v_h \ds 
      + \int_{\pO} (-\partial_n v_h) u_h \ds
    + \int_{\pO} \lambda_T u_h v_h \ds,
    \\
    l_h(v_h) & := \int_{\Omega} f v_h \dx + \int_{\pO}
    (-\partial_n v_h + \lambda_T v_h) g \ds.
  \end{align}
Here, $\partial_n$ denotes the partial derivative in normal direction
of the boundary of $\Omega$ and $\lambda_T$ denotes a element-wise
stabilization parameter.
The choice of the stabilization parameter is crucial.  
As the review of the error analysis in Section~\ref{sec:method} reveals,
the local stabilization parameter $\lambda_T$ has to be chosen sufficiently
large in order to guarantee the discrete coercivity of the bilinear form $a_h(\cdot,\cdot)$.
A common -- see \emph{e.g.,} \cite{Ruess2013,Ruess2014,Schillinger2015} -- strategy is to solve the eigenvalue problem:
\begin{equation}
  \label{eq:evprob}\begin{aligned}
    &\hspace{1.5cm}\mbox{Find } (u_\mu,\mu_T) \in V_h|_T^0 \times \mathbb{R}, \mbox{ such that:}\\
    &\int_{\pO \cap T} (\partial_n u_\mu) (\partial_n v_h)
    \, ds = \mu_T \int_{\Omega \cap T} \nabla u_\mu \nabla v_h \, ds
    \quad \mbox{for all } v_h \in V_h|_T^0,
  \end{aligned}
\end{equation}
on each element $T$ which has a non-empty intersection with $\pO$,
see for instance \cite{Embar2010}.
Here, $V_h|_T^0$ is the space of functions in $V_h|_T$ that are
$L^2(T)$-orthogonal to constants.  Let $N_T = \mathrm{dim}(V_h|_T)$,
then \eqref{eq:evprob} is an ($N_T-1$)-dimensional eigenvalue problem
with only non-negative eigenvalues.
Motivated by the forthcoming stability analysis, 
a suitable choice of the stabilization parameter is to set:
\begin{equation}\label{eq:choice}
  \lambda_T = 2\max\{\mu_T\} = 2 \max_{u_h \in V_h|_T^0} \frac{\int_{\pO \cap T} (\partial_n u_h)^2 \, ds}{\int_{\Omega \cap T} |\nabla u_h|^2 \, dx} > 0.
\end{equation}

\begin{remark}
  With techniques as in \cite{warburtonhesthaven03}, it can be shown that 
  $\lambda_T$ scales with  $|T \cap \pO|_{d-1}/|T \cap \Omega|_{d}$.
  Thus, shape-regularity implies that
  for a Nitsche-based formulation on a fitted and shape-regular mesh
  $\lambda_T \sim |T \cap \pO|_{d-1}/|T \cap \Omega|_{d} \sim h^{-1}$. 
  In contrast, in an unfitted method, the ratio
  $|T \cap \pO|_{d-1}/|T \cap \Omega|_{d}$ 
  depends not only on the cut element size but also highly on the cut configuration (see Figure~\ref{fig:cutcases}).
  In particular, a \emph{sliver} cut can lead
  to arbitrarily large values of $\lambda_T$ on relatively large parts of the boundary, depending on the
  thickness of the sliver.
  In Section~\ref{sec:numex} it is
  demonstrated how cut elements of approximately this shape can
  result in poor discretization errors.
\end{remark}
\begin{remark}
To generate and solve the discrete linear system associated with
  problem~\eqref{eq:formulation} several issues have to be resolved,
  \emph{i.e.,} sufficiently accurate numerical integration on $T \cap \Omega$
  and $T \cap \pO$ has to be provided, \emph{e.g.,} \cite{Yang2012,Abedian2013,Duczek2015,Fries2015,Lehrenfeld2016integrate,Joulaian2016,Kudela2016,Thiagarajan2016,Stavrev2016,Fries2016}, and a particular tuning
  of the linear algebra solver might be necessary,
  \emph{e.g.,} \cite{Lehrenfeld2016precon,Prenter2016}.
  All these aspects are important, but in this work
  we focus only of the choice of $\lambda_T$ and its effect on the numerical solution,
  assuming that numerical integration
  can be performed sufficiently accurate and applying a direct solver to
  solve the linear systems exactly.
\end{remark}

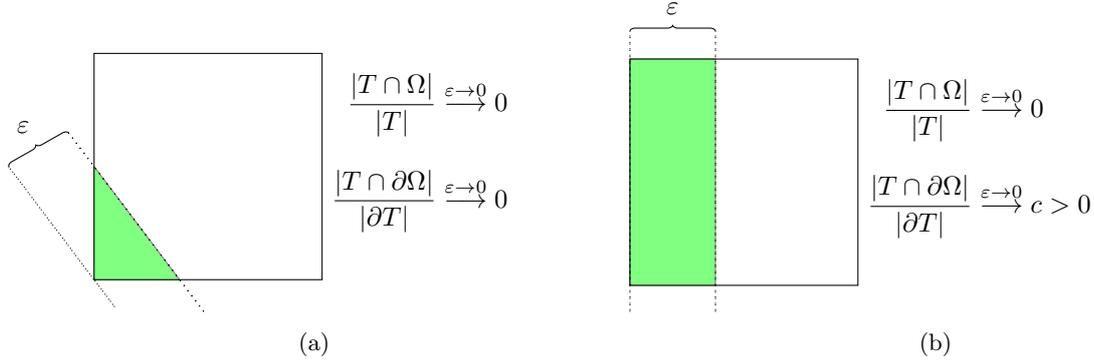
\begin{figure}
\begin{subfigure}[t]{.49\textwidth}
  \begin{tikzpicture}[scale=0.75]
    \begin{scope}
      \coordinate (V1) at (1,1); 
      \coordinate (V2) at (5,1);
      \coordinate (V3) at (5,5);
      \coordinate (V4) at (1,5);
      \coordinate (c1) at (2.5,1);
      \coordinate (c2) at (1,3);
      \coordinate (s1) at (-0.5,3);
      \coordinate (s2) at (1.375,0.5);

      \coordinate (t1) at ($ (c1) ! 1.3! (c2) $);
      \coordinate (t2) at ($ (c1) !-0.3! (c2) $);

      \draw [] (V2) -- (V3) -- (V4) -- (V1) --cycle; 
      \draw [fill=green,opacity=0.5] (V1) -- (c2) -- (c1) --cycle; 
      \draw [dotted] (s1) -- (s2) --cycle; 
      \draw [dotted] (t1) -- (t2) --cycle; 

      \draw [decorate,decoration={brace,amplitude=2pt},xshift=0pt,yshift=0pt]
      (s1) -- (t1) node [black,midway,yshift=0.3cm,xshift=-0.2cm]  {$\varepsilon$};

      \coordinate (Q) at (5,3.5); \node at (Q) [right] {
        \begin{minipage}{0.2\textwidth}
          \begin{align*}
            \frac{| T \cap \Omega|}{|T|} & \stackrel{\varepsilon \rightarrow 0}{\longrightarrow} 0 \\[2ex]
            \frac{| T \cap \pO|}{|\partial T|} & \stackrel{\varepsilon \rightarrow 0}{\longrightarrow} 0
          \end{align*}
        \end{minipage}
      };
    \end{scope}
  \end{tikzpicture}
  \caption{}
\end{subfigure}
\begin{subfigure}[t]{.49\textwidth}
  \begin{tikzpicture}[scale=0.75]
    \begin{scope}[shift={(9,0)}]
      \coordinate (V1) at (1,1); 
      \coordinate (V2) at (5,1);
      \coordinate (V3) at (5,5);
      \coordinate (V4) at (1,5);
      \coordinate (c1) at (2.5,1);
      \coordinate (c2) at (2.5,5);
      \coordinate (s1) at (1,5.5);
      \coordinate (s2) at (1,0.5);

      \coordinate (t1) at ($ (c1) ! 1.125! (c2) $);
      \coordinate (t2) at ($ (c1) !-0.125! (c2) $);

      \draw [] (V2) -- (V3) -- (V4) -- (V1) --cycle; 
      \draw [fill=green,opacity=0.5] (V1) -- (c1) -- (c2) -- (V4) --cycle; 
      \draw [dotted] (s1) -- (s2) --cycle; 
      \draw [dotted] (t1) -- (t2) --cycle; 

      \draw [decorate,decoration={brace,amplitude=2pt},xshift=0pt,yshift=0pt]
      (s1) -- (t1) node [black,midway,yshift=0.3cm]  {$\varepsilon$};

      \coordinate (Q) at (5,3.5); \node at (Q) [right] {
        \begin{minipage}{0.2\textwidth}
          \begin{align*}
            \frac{| T \cap \Omega|}{|T|} & \stackrel{\varepsilon \rightarrow 0}{\longrightarrow} 0 \\[2ex]
            \frac{| T \cap \pO|}{|\partial T|} & \stackrel{\varepsilon \rightarrow 0}{\longrightarrow} c > 0
          \end{align*}
        \end{minipage}
      };
    \end{scope}
  \end{tikzpicture}
  \caption{}
\end{subfigure}
  \caption{Different cut configurations. In (a) the part of the boundary in the element vanishes for vanishing volume, while for the \emph{sliver} case (b) the measure of the boundary
    part stays bounded from below as the volume goes to zero.} \label{fig:cutcases}
\end{figure}

\section{A simple error analysis} \label{sec:analysis}

In this section we perform an error analysis for the variational form in \eqref{eq:formulation}. We first define a mesh-dependent norm in which a best approximation property holds in Section~\ref{sec:energyNorm}. In Section~\ref{sec:approximability} we discuss the approximability in the mesh-dependent norm and the consequences for the discretization error in the $H^1(\Omega)$-norm.

\subsection{Best approximation in the mesh-dependent energy norm}\label{sec:energyNorm}

With the choice
of the stabilization parameter as in \eqref{eq:choice}, the stability
analysis is natural in the following mesh-dependent norm which we will
refer to as the \emph{energy norm}:
\begin{equation} \label{eq:enorm} \enorm{v}^2 := \Vert \nabla v
  \Vert_{\Omega}^2 + \Vert \lambda_T^{-\frac12} \partial_n v
  \Vert_{\pO}^2 + \Vert \lambda_T^{\frac12} v \Vert_{\partial
    \Omega}^2, \qquad v \in V_h \oplus
  H^2(\Omega).
\end{equation}
We note that the use of this norm in the error analysis is essentially
dictated by Nitsche's method and not by the model problem.  We obtain
coercivity and continuity in the following two lemmas.
\begin{lemma}[Coercivity]
  There holds:
  \begin{equation}
    a_h(u,u) \geq \frac15 \enorm{u}^2 \text{ for all } u \in V_h.
  \end{equation}
\end{lemma}
\begin{proof}
  Exploiting the eigenvalue characterization in the definition of
  $\lambda_T$ we obtain for arbitrary $\gamma > 0$:
  \begin{equation}
    \begin{aligned}
      a_h(u,u) & \geq \Vert \nabla u \Vert_{\Omega}^2 + \Vert \lambda_T^{\frac12} u \Vert_{\pO}^2 - 2 \int_{\pO} \left| \lambda_T^{-\frac12} \partial_n u \right| \, \left| \lambda_T^{\frac12} u \right| \, ds \\
      & \geq \Vert \nabla u \Vert_{\Omega}^2 + (1- \gamma) \Vert \lambda_T^{\frac12} u \Vert_{\pO}^2 -  \frac{1}{\gamma} \Vert \lambda_T^{-\frac12} \partial_n u \Vert_{\pO}^2 \\
      & \geq (1-\gamma) \Vert \nabla u \Vert_{\Omega}^2 + (1 - \gamma)
      \Vert \lambda_T^{\frac12} u \Vert_{\pO}^2 + (2
      \gamma - \frac{1}{\gamma}) \Vert \lambda_T^{-\frac12} \partial_n
      u \Vert_{\pO}^2 \geq c \enorm{u}^2,
    \end{aligned}
  \end{equation}
  with $c = 1 - \gamma^\ast$ for $\gamma^\ast>0$ such that
  $1-\gamma^\ast = 2 \gamma^\ast - \frac{1}{\gamma^\ast}$, which has
  only the positive solution $\gamma^\ast = \frac{1 + \sqrt{13}}{6}$
  so that $c > \frac15$.
\end{proof}
\begin{lemma}[Continuity]
  There holds:
  \begin{equation}
    a_h(u,v) \leq 2 \enorm{u} \enorm{v} \text{ for all } u,v \in V_h \oplus H^2(\Omega).
  \end{equation}
\end{lemma}
\begin{proof}
  With Cauchy-Schwarz we directly obtain:
  \begin{equation}
    \begin{aligned}
      \ \ & \!\!\!\! a_h(u,v) \!\leq\! \Vert \nabla u \Vert_{\Omega}
      \Vert \nabla v \Vert_{\Omega} \!+\!\Vert \lambda_T^{-\frac12}
      \!\partial_n u \Vert_{\pO} \Vert \lambda_T^{\frac12}
      v \Vert_{\pO} \!+\!\Vert \lambda_T^{-\frac12}
      \!\partial_n v \Vert_{\pO} \Vert \lambda_T^{\frac12}
      u \Vert_{\pO}
      \!+\! \Vert \lambda_T^{\frac12} u \Vert_{\pO} \Vert \lambda_T^{\frac12} v \Vert_{\pO} \\
      & \!\leq\! \Vert \nabla u \Vert_{\Omega} \Vert \nabla v
      \Vert_{\Omega}
      \!+\! \left( \Vert \lambda_T^{-\frac12} \!\partial_n u \Vert_{\pO}  + \Vert \lambda_T^{\frac12} u \Vert_{\pO} \right) \left( \Vert \lambda_T^{-\frac12} \!\partial_n v \Vert_{\pO}  + \Vert \lambda_T^{\frac12} v \Vert_{\pO} \right) \\
      & \leq \left( \Vert \nabla u \Vert_{\Omega}^2 + \big( \Vert
        \lambda_T^{-\frac12} \!\partial_n u \Vert_{\pO} +
        \Vert \lambda_T^{\frac12} u \Vert_{\pO} \big)^2
      \right)^{\frac12} \left( \Vert \nabla v
        \Vert_{\Omega}^2 +
        \big( \Vert \lambda_T^{-\frac12} \!\partial_n v \Vert_{\pO}  + \Vert \lambda_T^{\frac12} v \Vert_{\pO} \big)^2 \right)^{\frac12} \\
      & \leq \left( \Vert \nabla u \Vert_{\Omega}^2 + 2 \Vert
        \lambda_T^{-\frac12} \!\partial_n u \Vert_{\pO}^2
        + 2 \Vert \lambda_T^{\frac12} u \Vert_{\pO}^2
      \right)^{\frac12} \left( \Vert \nabla v
        \Vert_{\Omega}^2 + 2 \Vert \lambda_T^{-\frac12} \!\partial_n v
        \Vert_{\pO}^2 + 2 \Vert \lambda_T^{\frac12} v
        \Vert_{\pO}^2 \right)^{\frac12}.
    \end{aligned}
  \end{equation}
\end{proof}
Applying these lemmas and exploiting that the Nitsche formulation is
consistent in the sense of Galerkin orthogonality, we obtain C\'ea's
Lemma \cite{CEA}:
\begin{lemma}[C\'ea's Lemma]
  Let $u \in H^2(\Omega)$ be the solution to \eqref{eq:ellmodel} and
  $u_h \in V_h$ be the solution to
  \eqref{eq:formulation}. Then there holds:
  \begin{equation}\label{eq:cea}
    \enorm{u-u_h} \leq 11 \inf_{v_h \in V_h} \enorm{u-v_h}.
  \end{equation}
\end{lemma}
\begin{proof}
  For an arbitrary $v_h \in V_h$ there holds
  $ \enorm{u-u_h} \leq \enorm{u-v_h} + \enorm{v_h-u_h} $ and:
  \begin{equation}
    \enorm{v_h-u_h}^2 \leq 5 a_h(v_h-u_h,v_h-u_h) = 5 a_h(v_h-u,v_h-u_h) \leq 10 \enorm{v_h-u_h} \enorm{v_h-u},
  \end{equation}
  s.t.\ the claim directly follows.
\end{proof}
We note that in the proof we could not use coercivity for $u-u_h$, but
only for discrete functions. To this end we applied the triangle
inequality and were able to apply the coercivity estimate on the
discrete function $v_h-u_h$, which was then combined with Galerkin
orthogonality and continuity in the usual way.

This version of C\'ea's lemma suggests that Nitsche's method provides
a robust and quasi-optimal method.  The best approximation is,
however, only provided in the mesh-dependent norm
$\enorm{\cdot}$. In the next paragraph we discuss the discretization error of the best
approximation in this norm.

\subsection{Approximability of the solution in the mesh-dependent energy norm}\label{sec:approximability}
We consider the approximability of the solution $u$ with the finite element space $V_h$ in the energy norm, \emph{i.e.,} $\inf_{v_h \in V_h} \enorm{u -v_h}$, which is the r.h.s.\ in \eqref{eq:cea}.
The aim of this paragraph is to show that the approximability in the discrete energy norm is not robust with respect to the position of the domain boundary.
To this end, we split the energy norm, \emph{cf.}\ \eqref{eq:enorm}, into two parts with one part that is robust and one part which is not robust:
\begin{equation} \label{eq:enormest}
\inf_{v_h \in V_h} \!\! \enorm{u - v_h}^2
\! \geq
\inf_{v_h \in V_h} \!\!\left\{ \Vert \nabla (u - v_h) \Vert_{\Omega}^2 + \Vert \lambda_T^{-\frac12} \partial_n (u-v_h)\Vert_{\pO}^2  \right\}
+\!
\inf_{v_h \in V_h} \!\!\Vert \lambda_T^{\frac12}(u-v_h)\Vert_{\pO}^2.
\end{equation}
We assume that the solution $u$ is smooth, $u \in H^{k+1}(\Omega)$, with $k \geq 1$ denoting the order of the discretization. In the sequel we use the notation $\gtrsim$ for an inequality with a constant that is independent of the cut configuration.
For the first part of the r.h.s.\ in \eqref{eq:enormest} we take the idea from \cite{hansbo2002unfitted} to show that the solution can be optimally approximated in $V_h$ with respect to these parts of the norm independent of the cut position. Afterwards, we show that the remaining part is the crucial problem which results in the missing robustness. 

Let $\tilde{u} \in H^{k+1}(\widetilde{\Omega})$ be a continuous extension of $u$ to $\widetilde{\Omega}$ (the embedding domain, \emph{i.e.,} the union of all active elements), then we have:
$$
\Vert \nabla (u - v_h) \Vert_{\Omega} = \Vert \nabla (\tilde{u} - v_h) \Vert_{\Omega} \lesssim \Vert \nabla (\tilde{u} - v_h) \Vert_{\widetilde{\Omega}}, \quad \foralls v_h \in V_h.
$$
To bound the latter part we can apply a standard best approximation result for $V_h$ on the \emph{fitted} mesh $\widetilde{\mT}$. This estimate is robust in the position of the boundary, \emph{i.e.,} there is a $v_h \in V_h$ s.t. $\Vert \nabla(u - v_h) \Vert_{\Omega} \lesssim h^k \Vert \tilde{u} \Vert_{H^{k+1}(\widetilde{\Omega})} \lesssim h^k \Vert u \Vert_{H^{k+1}(\Omega)}$. 
Further, due to $\lambda_T^{-1} \lesssim h_T$ and a trace inequality, we have:
$$
\Vert \lambda_T^{-\frac12} \partial_n v \Vert_{\pO \cap T} \lesssim \Vert v \Vert_{H^1(\Omega \cap T)} + h_T \Vert v \Vert_{H^2(\Omega \cap T)}, \quad \foralls v \in H^2(\Omega \cap T).
$$
Applying this to $u-v_h$ gives similar results (independent of the position of the boundary),
hence $\Vert \lambda_T^{-\frac12} \partial_n (u-v_h) \Vert_\pO \lesssim h^k \Vert u \Vert_{H^{k+1}(\Omega)}$.

We now consider the second term in \eqref{eq:enormest} which is the crucial problem for the approximation of $u$ in the energy norm:
\begin{equation*}
  \Vert \lambda_T^{\frac12} \left( u - v_h \right) \Vert_{\pO}^2
  =
  \sum_{T \in \mT} \lambda_T \Vert u - v_h \Vert_{T \cap \pO}^2.
\end{equation*}
Let $h_{\Omega \cap T} := {\mathrm{diam}}(\Omega \cap T)$ and $h_{\pO \cap T} := {\mathrm{diam}}(\pO \cap T)$ be the characteristic lengths of the cut elements and their parts on the boundary. Furthermore, let $\tilde{v}_h$ be the best approximation of $u$ in $V_h$ in the energy norm. 
We have that $\Vert u - \tilde{v}_h \Vert_{T \cap \pO}^2$ is typically not zero and scales with $h_{\pO \cap T}^{2k+2+(d-1)}$, $d-1$ denoting the dimension of the boundary.
In the good case, if the \emph{cut element} $T \cap \Omega$ is shape regular, there holds $h_{\pO \cap T} \sim h_{\Omega \cap T} \sim \lambda_T^{-1}$ such that we would have $\lambda_T \Vert u - \tilde{v}_h \Vert_{T \cap \pO}^2 \leq c h_{\pO \cap T}^{2k+1+(d-1)} \leq c h_T^{2k+d}$ for some (bounded) constant $c > 0$ which is a robust bound for the element contribution of $T$.
In the bad case, if the \emph{cut element} is not shape regular however, such as for a sliver cut element in Figure~\ref{fig:cutcases}, it can occur that $\lambda_T \gg h_{\pO\cap T}^{-1}$, such that the contribution of a single element $\lambda_T \Vert u - \tilde{v}_h \Vert_{T \cap \pO}^2$ becomes unbounded.

We conclude that the last term in \eqref{eq:enormest} is unbounded and thus the same holds for the left hand side. Therefore the discretization error in the mesh-dependent energy norm, regardless of the best approximation property, can not be bounded from above. This missing robustness is essentially due to the possible existence of unfortunate cut configurations such as sliver cuts, which introduce values of the penalization parameter, $\lambda_T$, that are not bounded in terms of the reciprocal of the length of the boundary intersecting element $T$, $h_{\pO\cap T}$. 

The energy norm is very natural with respect to the (symmetric) Nitsche formulation.
Nevertheless, one is typically not interested in this norm but rather in the $H^1(\Omega)$-norm of the error.
To obtain a priori estimates for the error in the $H^1(\Omega)$-norm we
use a Poincar\'e-type inequality and combine it with the result in C\'ea's lemma:
\begin{equation} \label{eq:ceaagain} \Vert {u-u_h} \Vert_{H^1(\Omega)}
  \lesssim \enorm{u-u_h} \lesssim \inf_{v_h \in V_h}
  \enorm{u-v_h}.
\end{equation}
\renewcommand{\thefootnote}{$\dagger$}
Whilst the approximability in the energy norm is not robust with respect to the boundary position, these inequalities are independent of the boundary position\footnotemark. Therefore we can not give an optimal a priori error estimate for the $H^1(\Omega)$-norm.
\footnotetext{By \emph{e.g.,} Lemma B.63 in \cite{ErnGuermond} we have $\enorm{u}^2 \geq \Vert \nabla u \Vert_{\Omega}^2 +  \Vert \lambda_T^{\frac12} u \Vert_{\pO}^2 \geq \Vert \nabla u \Vert_{\Omega}^2 + \tilde{\lambda} \Vert u \Vert_{\pO}^2 \geq \delta^2 \Vert v\Vert_{H^1(\Omega)}^2$, with $\tilde{\lambda} \sim 1/h$ the lower bound for $\lambda_T$.}

Besides showing that the traditional method to show optimality does not hold for Nitsche's method, we can also give an intuitive derivation of why Nitche's method can yield bad approximations in the $H^1(\Omega)$-norm. To do this, we compare the energy and Sobolev norms. At first sight these norms are not even equivalent, as $\Vert \lambda^{-\frac12} \partial_n v \Vert_{\pO}$ is not bounded by the $H^1(\Omega)$-norm. This holds for all (even unsymmetric or fitted) Nitsche-type formulations however, and is not related to the problem of unbounded stabilization parameters we are targeting here. When we restrict ourselves to the space $W_h = V_h \oplus \langle u \rangle$, both norms are equivalent:
\begin{equation}\label{eq:normeq}
c \Vert v_h \Vert_{H^1(\Omega)} \leq \enorm{v_h} \leq C \Vert v_h \Vert_{H^1(\Omega)}, \quad \forall v_h \in W_h.
\end{equation}
To investigate the strength of this equivalence, we examine the constants $c$ and $C$ in \eqref{eq:normeq}. The first constant $0 < \delta \footnotemark < c < 1$ is bounded from below and above independent of $u$ or the cut configuration.
The upper bound of $c$ follows from:
\begin{equation}
c = \min_{v_h \in W_h} \frac{\enorm{v_h}}{\Vert v_h\Vert_{H^1(\Omega)}} \lesssim \min_{v_h \in \left. V_h \right|_{\pO}^\perp} \frac{\enorm{v_h}}{\Vert v_h\Vert_{H^1(\Omega)}} = \min_{v_h \in \left. V_h \right|_{\pO}^\perp} \frac{ | v_h |_{H^1}}{\Vert v_h\Vert_{H^1(\Omega)}} < 1,
\end{equation}
with $\left. V_h \right|_{\pO}^\perp$ the restriction of $V_h$ to functions that do not intersect the boundary, \emph{i.e.,}
$$
\left. V_h \right|_{\pO}^\perp= \left\{ v_h \in V_h | v_h = \partial_n v_h  = 0 { \ \mathrm{on}\ } \pO \right\}.
$$
The second constant in \eqref{eq:normeq}, $C$, does depend on the cut configuration:
\begin{equation}
C = \max_{v_h \in W_h} \frac{\enorm{v_h}}{\Vert v_h\Vert_{H^1(\Omega)}} \geq \max_{v_h \in V_h} \frac{\enorm{v_h}}{\Vert v_h\Vert_{H^1(\Omega)}}.
\label{eq:Cmax}
\end{equation}
Due to the finite dimensionality of $W_h$ we have $C<\infty$, but, in contrast to $c$, $C$ clearly depends on $\lambda$ and therefore on the cut configuration. The inequality in \eqref{eq:Cmax} disregards the possibility that $\Vert \lambda_T^{-\frac12} v_h \Vert_{\pO} \gg \Vert v_h \Vert_{H^1(\Omega)}$ for $v_h \in W_h$ due to possibly large normal derivatives of solution $u$, \emph{i.e.,} the general limitation for Nitsche-type methods we mentioned before. But also when this is not the case, $C$ will be large for large values of $\lambda$ however. Therefore, depending on the cut configuration, the equivalence between the energy norm and the $H^1(\Omega)$-norm can degenerate, such that the projections in the different norms can deviate from each other. With the solution to Nitsche's method closely resembling the projection in the energy norm, this may deviate substantially from an optimal approximation in $H^1(\Omega)$.

\begin{remark}
Using $c$ and $C$ we can also extend \eqref{eq:ceaagain}:
\begin{equation}
\Vert u_h - v_h\Vert_{H^1(\Omega)} \leq \frac{1}{c} \enorm{u_h - v_h} \leq 10 \frac{1}{c} \enorm{u - v_h} \leq 10 \frac{C}{c} \Vert u - v_h\Vert_{H^1(\Omega)}, \quad \forall v_h \in V_h,
\end{equation}
such that:
\begin{equation}
\Vert u - u_h\Vert_{H^1(\Omega)} \leq \left( 10 \frac{C}{c} + 1 \right) \inf_{v_h \in V_h} \Vert u - v_h\Vert_{H^1(\Omega)}.
\end{equation}
This seems like a satisfying result, but note that $C/c \gg 1$ when $\lambda \gg 1$.
\end{remark}

\section{Numerical examples}\label{sec:numex}

In this section we will demonstrate some examples of situations where
the previously discussed discretization leads to large values for the
stabilization parameter. The aim of this section is to illustrate the
problem described in Section~\ref{sec:analysis} and to show that it is
not just a problem in the a priori analysis, but really yields large
discretization errors in practice.

The numerical integration on elements that only partially intersect the domain is performed using the bisection-based tessellation scheme proposed in \cite{Verhoosel2015} with a maximal refinement depth of two. This leads to an exact representation of the domain for the first three examples and a reasonable good geometric approximation for the other examples. To make sure the results of these last examples are not affected by geometrical errors, we adapt the boundary conditions to the approximated geometry. On the approximated geometry we apply Gauss quadrature such that functions of order $2k+1$ are integrated exactly, in order to achieve an accurate integration of the  force terms and the boundary conditions.

\subsection{Example 1: An overlapping square on a triangular and quadrilateral mesh}
The simplest example of a discretization with \emph{sliver} cuts as
described in the Section~\ref{sec:method} is the squared domain
$\Omega = (-1-\varepsilon,1+\varepsilon)^2$ with
$0 < \varepsilon \ll 1$ on a structured grid that aligns with $\Omega$
for $\varepsilon = 0$. The domain and the grid are shown in
Figure~\ref{fig:testcase1mesh} (for the triangular mesh), in which it is clearly visible that the
cut elements become of arbitrarily large aspect ratios for
$\varepsilon \rightarrow 0$. 

\subsubsection*{Triangular grid}
We use a \emph{triangular} mesh of isosceles right triangles with legs of length $h = 1/K$ with $K=16$ and continuous piecewise linear basis functions.  
We apply an element local stabilization parameter by the procedure described in
Section~\ref{sec:method}. We set up the problem such that the
analytical solution equals $u = \sin(\pi x) + \sin(\pi y)$ and impose
Dirichlet conditions through the previously discussed Nitsche method
on all boundaries.

\begin{figure}[h!]
  \begin{center}
    \begin{subfigure}[t]{.49\textwidth}
      \begin{center}
        \begin{tikzpicture}
          [ spy using outlines={rectangle, inner sep=0pt, width=1.8cm,
            height=5.5cm, very thick,black,magnification=3.5, connect
            spies,opacity=1.0}, scale=0.5 ] \def\dw{1} \def\dr{10}
          \def\da{0.92} 
          \def\db{1.1} 
          \def\dc{3} 

          \draw[ultra thin,red] (\da*\dw,-0.4*\dw) -- (\da*\dw,10.4*\dw);
          \draw[ultra thin,red] (\dw,-0.4*\dw) -- (\dw,10.4*\dw);

          \draw[ultra thin,draw=black,fill=green,opacity=0.65]
          (10*\dw-\da*\dw,10*\dw-\da*\dw)
          -- (10*\dw-\da*\dw,\da*\dw)
          -- (\da*\dw,\da*\dw)
          -- (\da*\dw,10*\dw-\da*\dw)
          -- cycle ;

          \fill[red,opacity=0.2] (\da*\dw,-0.4*\dw) -- (\da*\dw,10.4*\dw)
          -- (\dw,10.4*\dw) -- (\dw,-0.4*\dw); \draw [draw=none]
          (\da*\dw,-0.4*\dw) -- (\dw,-0.4*\dw) node
          [pos=0.5,anchor=north,yshift=0.03cm,scale=0.8] {$\varepsilon$};

          \foreach \i in {0,...,9} { \foreach \j in {0,...,9} { \draw
              [ultra thin,opacity=1.0] (\i*\dw,\j*\dw) --
              (\i*\dw+\dw,\j*\dw+\dw); \draw [ultra thin,opacity=1.0]
              (\i*\dw+\dw,\j*\dw) -- (\i*\dw,\j*\dw+\dw); \draw [ultra
              thin,opacity=1.0] (\i*\dw,\j*\dw+0.5*\dw) --
              (\i*\dw+\dw,\j*\dw+0.5*\dw); \draw [ultra thin,opacity=1.0]
              (\i*\dw+0.5*\dw,\j*\dw) -- (\i*\dw+0.5*\dw,\j*\dw+\dw); } }
          \foreach \i in {0,...,9} { \foreach \j in {0,...,10} { \draw
              [ultra thin,opacity=1.0] (\i*\dw,\j*\dw) --
              (\i*\dw+\dw,\j*\dw); } } \foreach \i in {0,...,10} {
            \foreach \j in {0,...,9} { \draw [ultra thin,opacity=1.0]
              (\i*\dw,\j*\dw) -- (\i*\dw,\j*\dw+\dw); } }

          \coordinate (spyat) at (0.8*\dw,0.6*\dw); \coordinate
          (showspyat) at (-2.3*\dw,4.5*\dw); \spy on (spyat) in node
          [fill=white] at (showspyat);
        \end{tikzpicture}
      \end{center}
      \caption{Mesh and domain $\Omega$.}
      \label{fig:testcase1mesh}
    \end{subfigure}
    \hfill
    \begin{subfigure}[t]{.45\textwidth}
      \begin{center}
        \includegraphics[trim=4mm 4mm 4mm 2mm, width=0.995\textwidth]{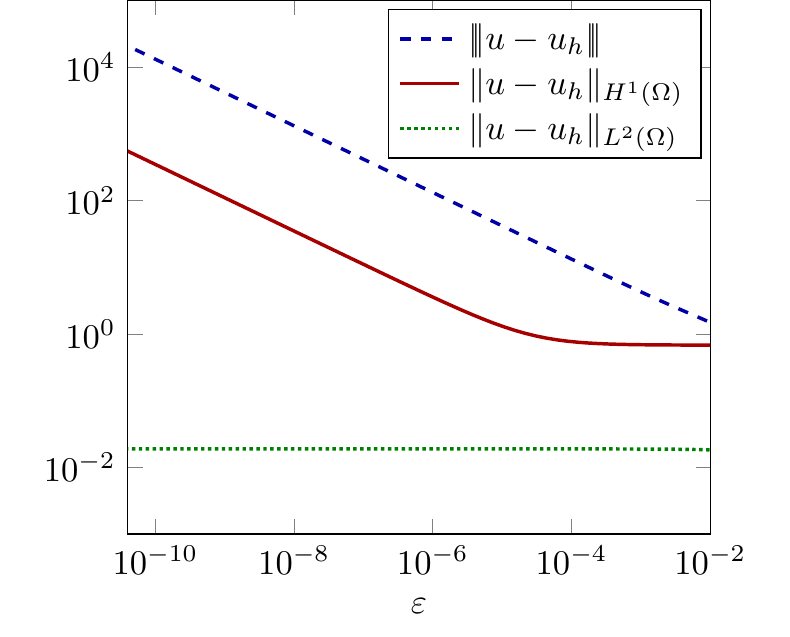} \\
      \end{center}
      \caption{Discretization errors.}
      \label{fig:testcase1results}
    \end{subfigure}
  \end{center}
  \vspace*{-0.6cm}
  \caption{Example 1 with a triangular grid.}
  \vspace*{-0.1cm}
  \label{fig:testcase1}
\end{figure}
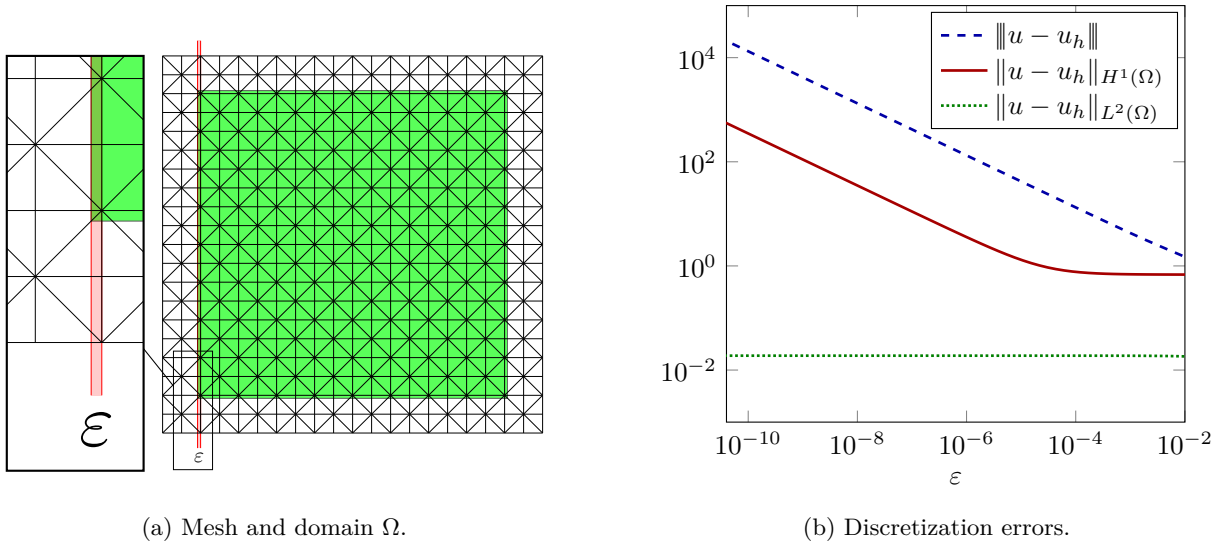

The resulting discretization errors in the mesh-dependent energy norm,
the $H^1(\Omega)$-norm and the $L^2(\Omega)$-norm are given in
Figure~\ref{fig:testcase1results}. From the results it is clearly visible that
for $\varepsilon \rightarrow 0$ the discretization error degenerates in both
the energy as the $H^1(\Omega)$-norm.  

\subsubsection*{Quadrilateral grid}
We repeat the example on a \emph{quadrilateral} grid with the same nodal positions as the triangular grid, \emph{i.e.,} the mesh size $h=1/K$ with $K=16$, see Figure~\ref{fig:testcase2mesh}.
Piecewise bilinear elements are used which lead to the same number of degrees of freedom. The results are shown in Figure~\ref{fig:testcase2results}, from which we can observe the same tendencies for the convergence $\varepsilon \rightarrow 0$. 

\begin{figure}[h!]
  \begin{center}
    \begin{subfigure}[t]{.49\textwidth}
      \begin{center}
        \begin{tikzpicture}
          [ spy using outlines={rectangle, inner sep=0pt, width=1.8cm,
            height=5.5cm, very thick,black,magnification=3.5, connect
            spies,opacity=1.0}, scale=0.5 ] \def\dw{1} \def\dr{10}
          \def\da{0.92} 
          \def\db{1.1} 
          \def\dc{3} 

          \draw[ultra thin,red] (\da*\dw,-0.4*\dw) -- (\da*\dw,10.4*\dw);
          \draw[ultra thin,red] (\dw,-0.4*\dw) -- (\dw,10.4*\dw);

          \draw[ultra thin,draw=black,fill=green,opacity=0.65]
          (10*\dw-\da*\dw,10*\dw-\da*\dw)
          -- (10*\dw-\da*\dw,\da*\dw)
          -- (\da*\dw,\da*\dw)
          -- (\da*\dw,10*\dw-\da*\dw)
          -- cycle ;

          \fill[red,opacity=0.2] (\da*\dw,-0.4*\dw) -- (\da*\dw,10.4*\dw)
          -- (\dw,10.4*\dw) -- (\dw,-0.4*\dw); \draw [draw=none]
          (\da*\dw,-0.4*\dw) -- (\dw,-0.4*\dw) node
          [pos=0.5,anchor=north,yshift=0.03cm,scale=0.8] {$\varepsilon$};

          \foreach \i in {0,...,9} { \foreach \j in {0,...,9} {
              \draw [ultra thin,opacity=1.0] (\i*\dw,\j*\dw+0.5*\dw) --
              (\i*\dw+\dw,\j*\dw+0.5*\dw); \draw [ultra thin,opacity=1.0]
              (\i*\dw+0.5*\dw,\j*\dw) -- (\i*\dw+0.5*\dw,\j*\dw+\dw); } }
          \foreach \i in {0,...,9} { \foreach \j in {0,...,10} { \draw
              [ultra thin,opacity=1.0] (\i*\dw,\j*\dw) --
              (\i*\dw+\dw,\j*\dw); } } \foreach \i in {0,...,10} {
            \foreach \j in {0,...,9} { \draw [ultra thin,opacity=1.0]
              (\i*\dw,\j*\dw) -- (\i*\dw,\j*\dw+\dw); } }

          \coordinate (spyat) at (0.8*\dw,0.6*\dw); \coordinate
          (showspyat) at (-2.3*\dw,4.5*\dw); \spy on (spyat) in node
          [fill=white] at (showspyat);
        \end{tikzpicture}
      \end{center}
      \caption{Mesh and domain $\Omega$.}
      \label{fig:testcase2mesh}
    \end{subfigure}
    \hfill
    \begin{subfigure}[t]{.45\textwidth}
      \begin{center}
        \includegraphics[trim=4mm 4mm 4mm 2mm, width=0.995\textwidth]{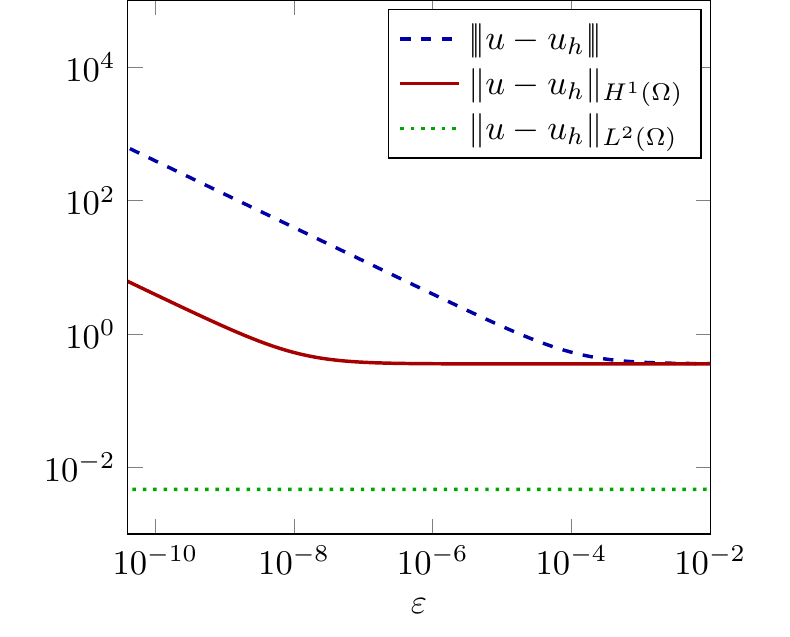} \\
      \end{center}
      \caption{Discretization errors.}
      \label{fig:testcase2results}
    \end{subfigure}
  \end{center}
  \vspace*{-0.6cm}
  \caption{Example 1 with a quadrilateral grid.}
  \vspace*{-0.1cm}
  \label{fig:testcase2}
\end{figure}
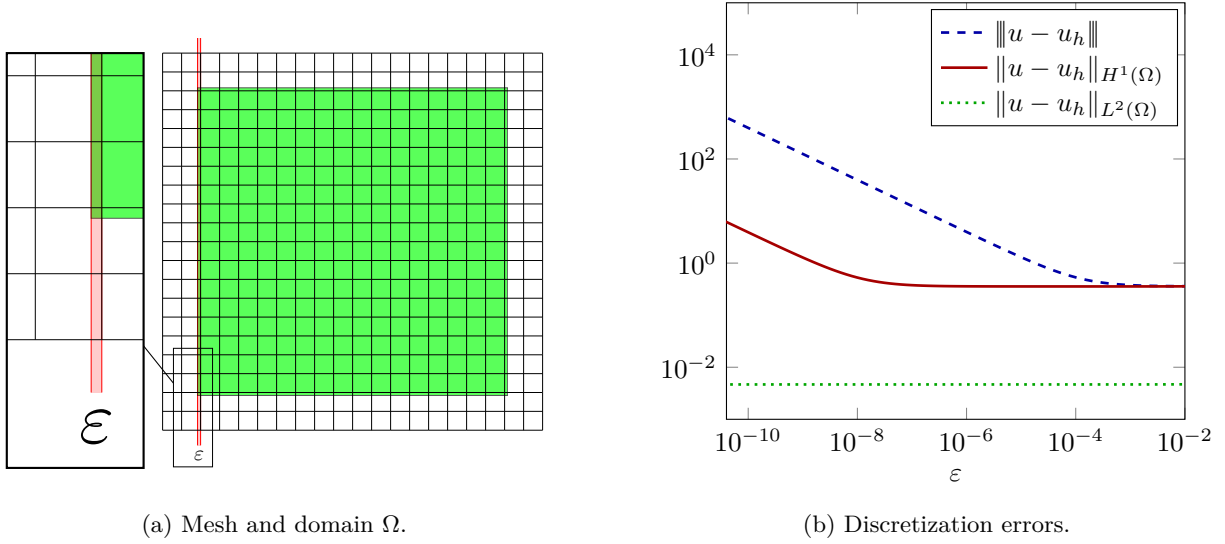

We observe that the impact of a small $\varepsilon$ is less severe on the quadrilateral grid compared to the triangular grid. The error is several orders of magnitude smaller and the diverging effect only starts at a smaller value of $\varepsilon$. This difference is explained in the next paragraph, which is not a rigorous proof, but gives an intuitive illustration of the nature of the difference. It further serves as a motivation for the design of Examples 2 and 3.

\subsubsection*{Discussion of the difference between triangular and quadrilateral grids}
In Section~\ref{sec:energyNorm} we have seen that the numerical solution closely approximates the energy norm projection of the analytical solution on the finite dimensional function space. For large values of $\lambda_T$, the contribution of normal derivatives to the energy norm diminishes and we can consider the energy norm to be a weighted combination of an $L^2(\pO)$-norm on the boundary and an $H_0^1(\Omega)$-seminorm in the interior. For increasing $\lambda_T \propto \varepsilon^{-1}$, the balance between these weighted norms is lost, and the energy norm projection basically consists of an $L^2(\pO)$-norm projection for all \dofs\ supported on the boundary and an $H_0^1(\Omega)$-seminorm projection in the interior for the remaining \dofs. Hence, in the limit of $\varepsilon \rightarrow 0$, the numerical solution $u_h$ satisfies:
\begin{equation}
u_h|_{\pO} \stackrel{\varepsilon \rightarrow 0}{\longrightarrow} u_h^\ast :=  \operatornamewithlimits{argmin}_{v_h\in V_h|_{\pO}} \Vert \lambda_T^{\frac12} (v_h - g) \Vert_{L^2(\pO)}^2.
\end{equation}
Hence $M := \mathrm{dim}(V_h|_{\pO})$ \dofs\ are fixed independently of the remaining part of the variational formulation. Note that $M=16K+7$ for the triangular grid (corresponding to the number of dots minus one in Figure~\ref{fig:triangDofs}) and that $M=8K+8$ for the quadrilateral grid (corresponding directly to the number of dots in Figure~\ref{fig:quadDofs}). Besides constraining the boundary values, the $L^2(\pO)$-norm projection on the boundary can also affect the normal derivatives or even the complete numerical solution in the cut elements however. Let $N$ denote the number of basis functions with support on $\pO$ (squares in Figure~\ref{fig:explanation}), then we have $N = 16K+8$ for both meshes. Note that $N$ is also equal to the number of basis functions supported on cut elements. For the triangular grid we have $M = N-1$, which implies that in the limit of $\varepsilon \to 0$ only one \dof\ in all cut elements together is not constrained by the boundary condition. For the quadrilateral mesh we have $M = 1/2 N + 4$, such that only approximately half of the \dofs\ supported on the boundary is constrained by the boundary condition. Along the straight parts of the boundary, the normal derivatives can therefore be set by the $H_0^1(\Omega)$-seminorm projection in the interior independent of the boundary condition. Only in the corners the normal derivatives are constrained by the boundary condition. As the problematic area is much smaller compared to the case with the triangular mesh, the impact of large stabilization parameters is -- although asymptotically the same -- less severe. 
\begin{figure}[h!]
  \begin{center}
    \begin{subfigure}[t]{.49\textwidth}
      \begin{center}
        \includegraphics[width=0.9\textwidth]{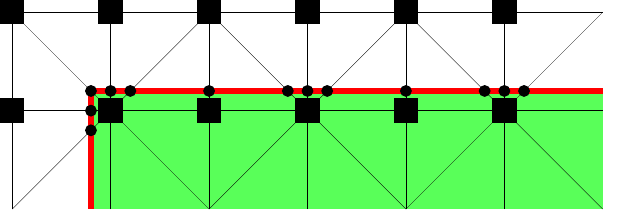}
        \caption{Triangular grid.\label{fig:triangDofs}}
      \end{center}
    \end{subfigure}
    \hfill
    \begin{subfigure}[t]{.49\textwidth}
      \begin{center}
        \includegraphics[width=0.9\textwidth]{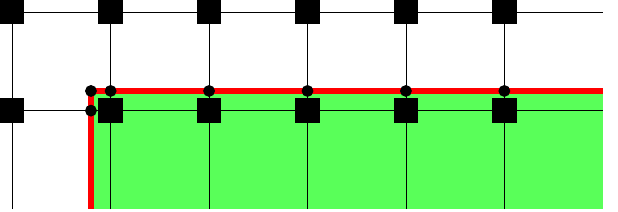}
      \caption{Quadrilateral grid.\label{fig:quadDofs}}
      \end{center}
    \end{subfigure}
  \end{center}
  \vspace{-.6cm}
  \caption{Sketch of \dofs \ of $V_h$ (squares) and \dofs \ of $V_h|_{\pO}$ (dots).}
  \label{fig:explanation}
\end{figure}

From Figures~\ref{fig:testcase1results} and \ref{fig:testcase2results} it is clear that small values of $\varepsilon$ do not only yield bad approximations, but actually cause diverging discretization errors. To explain this effect, we note that the boundary data $g$ has magnitudes $\sim 1$ whilst the basis functions that are only supported on cut elements have values of magnitude $\sim \varepsilon$ on the boundary. The $L^2(\pO)$-norm projection of the boundary condition therefore constrains the coefficients of these basis functions to a value of magnitude $\sim \varepsilon^{-1}$. This results in gradients in the cut elements that scale with $\varepsilon^{-1}$, which has also been observed by peak stresses in immersed elasticity problems \cite{Miegroet2007,Jiang2015}. These large gradients introduce an $\mathcal{O}(\varepsilon^{-\frac12})$ error in the $H^1(\Omega)$-norm, \emph{i.e.,} the sum of the volumes of the cut elements has a magnitude of $\sim \varepsilon$, whilst the error in the gradient squared has a value of magnitude $\sim \varepsilon^{-2}$. We notice that the $L^2(\Omega)$-norm error seems to be unaffected by the sliver cuts.

\subsection{Example 2: An overlapping square with mixed boundary conditions}
In this example we reuse the quadrilateral mesh and try to remove the problematic corner regions. To this end, we pose Neumann conditions on the left and right boundaries, $\pO_N =  \{|x| = 1+\varepsilon\}$, and Dirichlet conditions on the lower and upper boundaries, $\pO_D = \{ |y| = 1 + \varepsilon \}$.
Our hypothesis is that this will yield good approximation results in the $H^1(\Omega)$-norm for the quadrilateral grid, because in that case an $L^2(\pO_D)$-norm projection on the Dirichlet boundary will constrain $M = 4K+6$ \dofs\ whilst the number of basis functions supported on the Dirichlet boundary equals $N = 2M$.
On all elements, including the corner elements, there is exactly one constraint for every two degrees of freedom.
This allows to approximate the boundary values without affecting the normal derivatives. 

The results are visible in Figure~\ref{fig:results3} and indeed show
that in this situation the error in the $H^1(\Omega)$-norm is robust
with respect to $\varepsilon$, even though
the error in the mesh-dependent energy norm still scales with
$\varepsilon^{-\frac12}$. This is because also the best possible approximation of the boundary conditions has an error, such that the energy norm error increases for increasing $\lambda$. Furthermore, we notice that the errors in the
energy norm are reduced by approximately a factor $\sqrt{2}$, which is
accounted for by the reduction of the size of the Dirichlet boundary
by a factor $2$.
\begin{figure}[h!]
  \begin{center}
    \includegraphics[trim=4mm 4mm 4mm 2mm, width=0.9\textwidth]{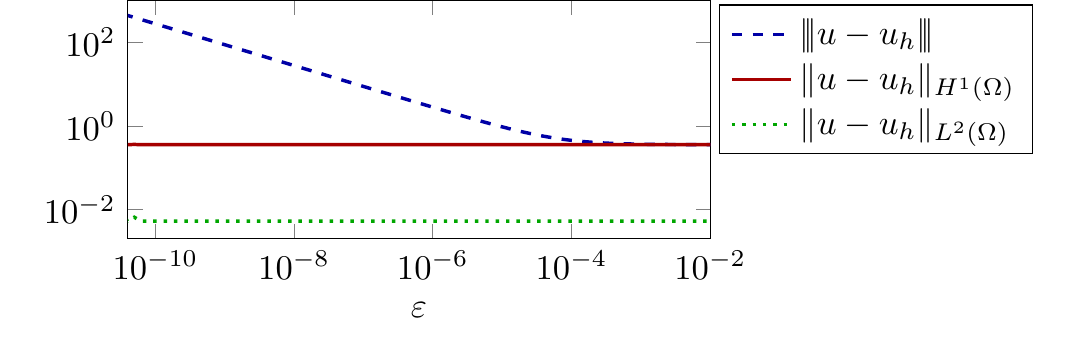} \\
  \end{center}
  \caption{Discretization errors for Example 2.}
  \label{fig:results3}
\end{figure}

\subsection{Example 3: A tilted square with mixed boundary conditions}
In this example we demonstrate that Example 2, be it useful for academic purposes, does not represent a general case for straight boundaries and tensor product grids. The reason that the second example yields robust approximation errors is that the Dirichlet boundary is parallel to a grid line, such that $N=2M$, \emph{i.e.,} the number of \dofs\ constrained by the $L^2(\pO_D)$-norm projection is exactly right for linear bases. This is not generally the case for immersed methods however. In this third example we do not consider the square with boundaries $\pO_N = \{|x| = 1+\varepsilon\}$ and $\pO_D = \{|y| = 1+\varepsilon\}$, but consider a kinked domain with boundaries $\pO_N = \{|x - \varepsilon y| = 1\}$ and $\pO_D = \{|y - \varepsilon x| = 1\}$ as can be seen in Figure~\ref{fig:testcase4mesh}. The (quadrilateral) grid and the analytical solution are the same as in the previous examples.

\begin{figure}
  \begin{center}
    \begin{subfigure}[t]{.59\textwidth}
    \begin{minipage}{.9\textwidth}
      \begin{center}
      \begin{tikzpicture}
        [ spy using outlines={rectangle, inner sep=0pt, width=1.5cm,
          height=4.5cm, very thick,black,magnification=3, connect
          spies ,opacity=1.0}, scale=0.5 ] \def\dw{1} \def\dr{10}
        \def\daa{0.9130435} 
        \def\dab{1.0740741} 
        \def\db{1.1} 
        \def\dc{3} 

        \draw[ultra thin,red] (\daa*\dw,-0.4*\dw) --
        (\daa*\dw,10.4*\dw); \draw[ultra thin,red] (\dw,-0.4*\dw) --
        (\dw,10.4*\dw);

        \draw[ultra thin,red] (10*\dw-\dab*\dw,-0.4*\dw) --
        (10*\dw-\dab*\dw,10.4*\dw); \draw[ultra thin,red]
        (10*\dw-\dw,-0.4*\dw) -- (10*\dw-\dw,10.4*\dw);

        \draw[ultra thin,draw=black,fill=green,opacity=0.65]
        (10*\dw-\daa*\dw,10*\dw-\daa*\dw)
        -- (10*\dw-\dab*\dw,\dab*\dw)
        -- (\daa*\dw,\daa*\dw)
        -- (\dab*\dw,10*\dw-\dab*\dw)
        -- cycle ;

        \fill[red,opacity=0.2] (\daa*\dw,-0.4*\dw) --
        (\daa*\dw,10.4*\dw) -- (\dw,10.4*\dw) -- (\dw,-0.4*\dw);
        \fill[red,opacity=0.2] (10*\dw-\dab*\dw,-0.4*\dw) --
        (10*\dw-\dab*\dw,10.4*\dw) -- (9*\dw,10.4*\dw) --
        (9*\dw,-0.4*\dw);
        \draw [draw=none] (\daa*\dw,-0.4*\dw) -- (\dw,-0.4*\dw) node
        [pos=0.5,anchor=north,yshift=0.03cm,scale=0.35]
        {$\frac{\varepsilon}{1-\varepsilon}$}; \draw [draw=none]
        (10*\dw-\dab*\dw,-0.4*\dw) -- (9*\dw,-0.4*\dw) node
        [pos=0.5,anchor=north,yshift=0.03cm,scale=0.35]
        {$\frac{\varepsilon}{1+\varepsilon}$};

        \foreach \i in {0,...,9} { \foreach \j in {0,...,9} {
            \draw [ultra thin,opacity=1.0] (\i*\dw,\j*\dw+0.5*\dw) --
            (\i*\dw+\dw,\j*\dw+0.5*\dw); \draw [ultra
            thin,opacity=1.0] (\i*\dw+0.5*\dw,\j*\dw) --
            (\i*\dw+0.5*\dw,\j*\dw+\dw); } } \foreach \i in {0,...,9}
        { \foreach \j in {0,...,10} { \draw [ultra thin,opacity=1.0]
            (\i*\dw,\j*\dw) -- (\i*\dw+\dw,\j*\dw); } } \foreach \i in
        {0,...,10} { \foreach \j in {0,...,9} { \draw [ultra
            thin,opacity=1.0] (\i*\dw,\j*\dw) -- (\i*\dw,\j*\dw+\dw);
          } }

        \coordinate (spyat) at (0.8*\dw,0.6*\dw); \coordinate
        (showspyat) at (-2.3*\dw,4*\dw); \spy on (spyat) in node
        [fill=white] at (showspyat);

        \coordinate (spyatb) at (10*\dw-1.06*\dw,0.6*\dw); \coordinate
        (showspyatb) at (10*\dw+2.3*\dw,4*\dw); \spy on (spyatb) in
        node [fill=white] at (showspyatb);
      \end{tikzpicture}
    \end{center}
    \end{minipage}
      \caption{Mesh and domain $\Omega$.}
      \label{fig:testcase4mesh}
    \end{subfigure}
    \hfill
    \begin{subfigure}[t]{.39\textwidth}
    \begin{minipage}{\textwidth}
      \begin{center}
        \includegraphics[trim=4mm 4mm 4mm 2mm, width=0.995\textwidth]{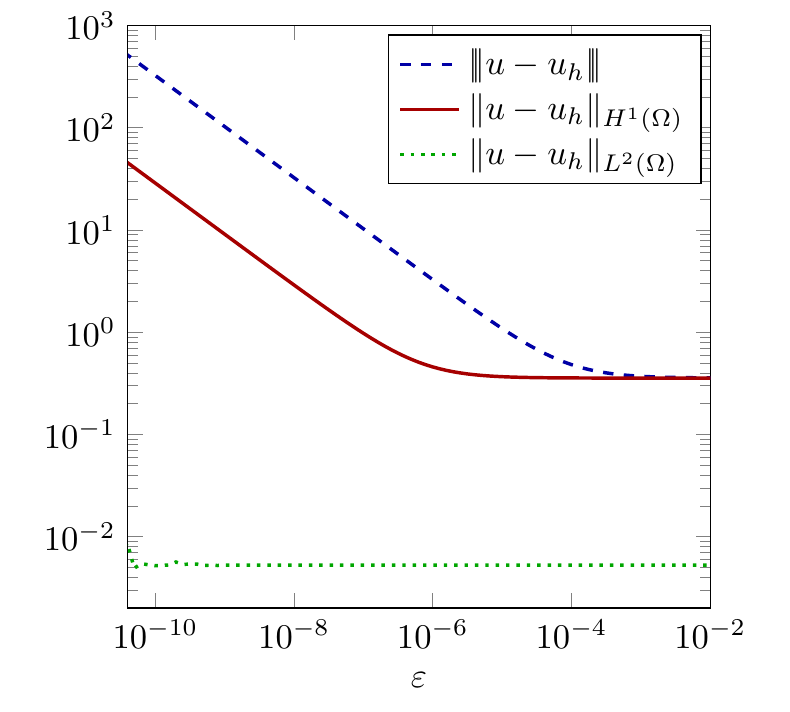}
      \end{center}
    \end{minipage}
      \vspace{.3cm}\caption{Discretization errors.}
      \label{fig:testcase4results}
    \end{subfigure}
  \end{center}
  \vspace*{-0.6cm}
  \caption{Example 3.}
  \vspace*{-0.1cm}
  \label{fig:testcase4}
\end{figure}
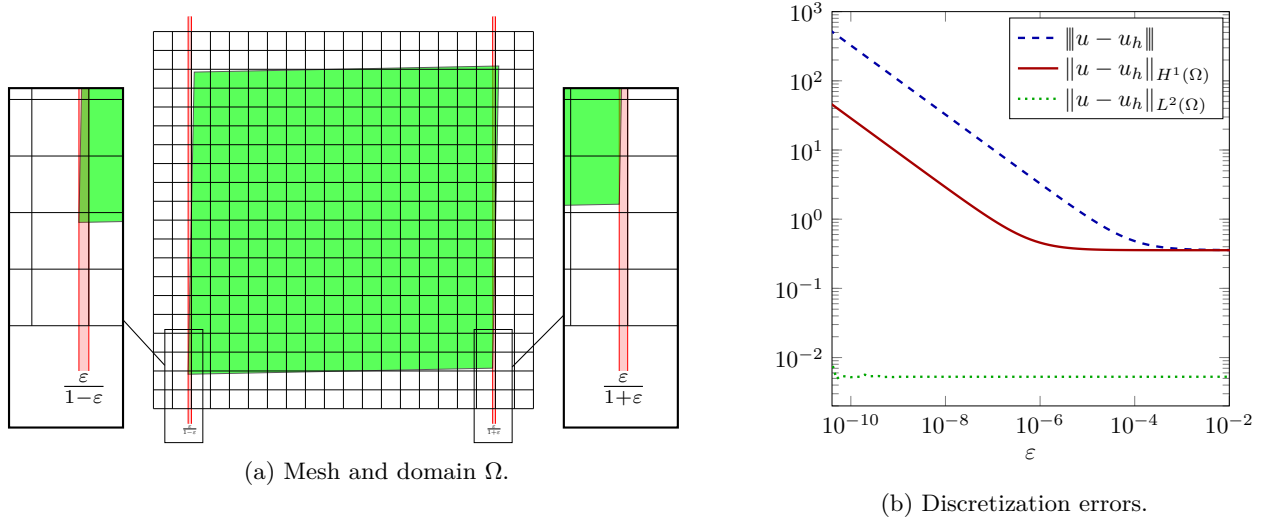

In this configuration we have $M=8K+6$ constraints in the $L^2(\pO_D)$-norm projection on the Dirichlet boundary and $N=8K+10$ basis functions that are supported on the Dirichlet boundary. Even though there is a kernel of dimension $4$ in the projection of the boundary conditions on the nodes that are supported on the Dirichlet boundary, basis functions with values of magnitude $\sim \epsilon$ on the boundary will be constrained to nodal values of magnitude $\sim \varepsilon^{-1}$. The results in Figure~\ref{fig:testcase4results} indeed show similar behavior to the first example.

\subsection{Example 4: A rounded square ($\Vert \cdot \Vert_8$-norm unit ball)}
The previous examples consisted of domains with piecewise straight
boundaries. Because unfitted methods are generally used for complex
domains, this example involves a curved boundary that creates sliver-like cuts which can occur in many realistic applications. The
boundary of the domain is defined by the function
$x^8+y^8 = (1+\varepsilon)^8$ and Dirichlet conditions are imposed on
the complete boundary. A sketch of the domain is visible in
Figure~\ref{fig:testcase5mesh}.

\begin{figure}[h!]
  \begin{center}
    \begin{subfigure}[t]{.49\textwidth}
      \begin{center}
      \begin{tikzpicture}
      [ spy using outlines={rectangle, inner sep=0pt, width=1.1cm,
        height=4.5cm, very thick,black,magnification=2.3, connect
        spies ,opacity=1.0}, scale=0.5 ] \def\dw{1} \def\dr{10}
      \def\da{0.92} 
      \def\db{1.1} 
      \def\dc{3} 

      \begin{axis}[
        width=11.6cm, height=11.6cm,
        xmin=-1.25,xmax=1.25, ymin=-1.25,ymax=1.25, grid=none,
        xtick=\empty, ytick=\empty, axis lines=none, scale=1.0 ]
        \addplot [ultra
        thin,domain=0:360,samples=300,fill=green,opacity=0.65] ( {
          1.02*cos(x)/( ((cos(x))^8 + (sin(x))^8)^(1/8)) }, {
          1.02*sin(x)/( ((cos(x))^8 + (sin(x))^8)^(1/8)) } );
      \end{axis}

      \draw[ultra thin,red] (\da*\dw,-0.4*\dw) -- (\da*\dw,10.4*\dw);
      \draw[ultra thin,red] (\dw,-0.4*\dw) -- (\dw,10.4*\dw);

      \fill[red,opacity=0.2] (\da*\dw,-0.4*\dw) -- (\da*\dw,10.4*\dw)
      -- (\dw,10.4*\dw) -- (\dw,-0.4*\dw); \draw [draw=none]
      (\da*\dw,-0.4*\dw) -- (\dw,-0.4*\dw) node
      [pos=0.5,anchor=north,yshift=0.03cm,scale=0.8] {$\varepsilon$};

      \foreach \i in {0,...,9} { \foreach \j in {0,...,9} {
          \draw [ultra thin,opacity=1.0] (\i*\dw,\j*\dw+0.5*\dw) --
          (\i*\dw+\dw,\j*\dw+0.5*\dw); \draw [ultra thin,opacity=1.0]
          (\i*\dw+0.5*\dw,\j*\dw) -- (\i*\dw+0.5*\dw,\j*\dw+\dw); } }
      \foreach \i in {0,...,9} { \foreach \j in {0,...,10} { \draw
          [ultra thin,opacity=1.0] (\i*\dw,\j*\dw) --
          (\i*\dw+\dw,\j*\dw); } } \foreach \i in {0,...,10} {
        \foreach \j in {0,...,9} { \draw [ultra thin,opacity=1.0]
          (\i*\dw,\j*\dw) -- (\i*\dw,\j*\dw+\dw); } }

      \coordinate (spyat) at (0.8*\dw,1*\dw); \coordinate (showspyat)
      at (-2.3*\dw,4.5*\dw); \spy on (spyat) in node [fill=white] at
      (showspyat);
    \end{tikzpicture}
    \end{center}
      \caption{Mesh and domain $\Omega$.}
      \label{fig:testcase5mesh}
    \end{subfigure}
  \hfill
  \begin{subfigure}[t]{.49\textwidth}
    \begin{center}
      \includegraphics[trim=4mm 4mm 4mm 2mm,
      width=0.995\textwidth]{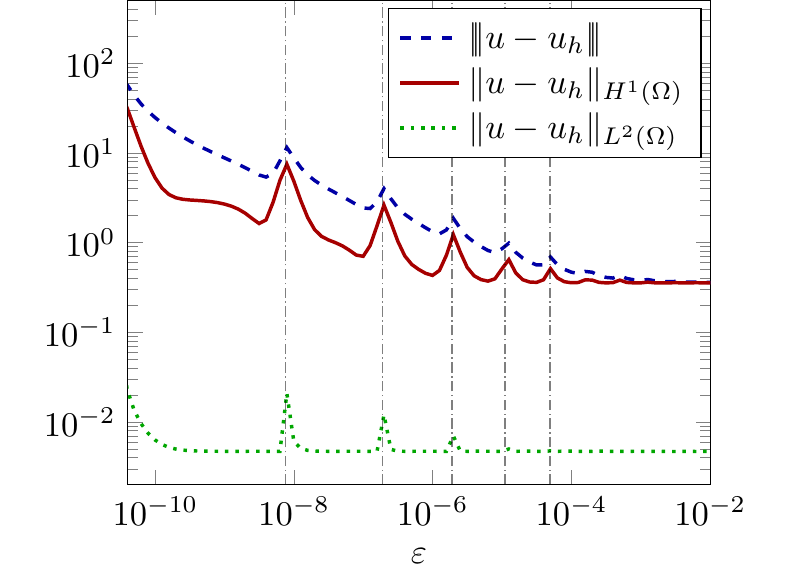}
      \end{center}
      \caption{Discretization errors.}
      \label{fig:testcase5results}
      \label{fig:results5}
    \end{subfigure}
  \end{center}
  \vspace*{-0.6cm}
  \caption{Example 4.}
  \vspace*{-0.1cm}
  \label{fig:testcase5}
\end{figure}
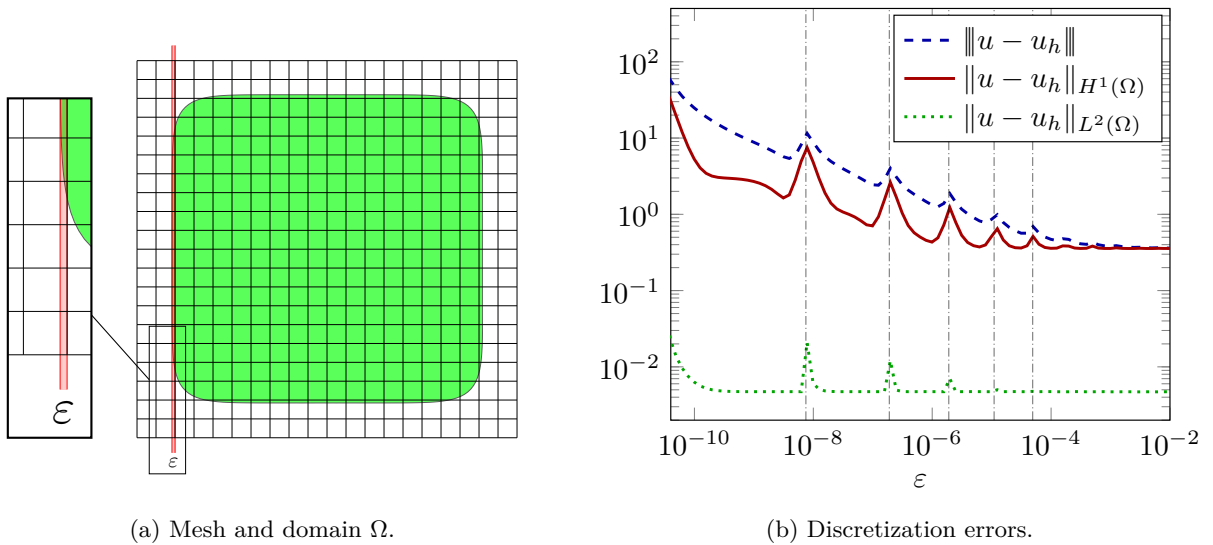

The results are visible in Figure~\ref{fig:testcase5results}. From the results
we note that for this example not even the error in the
$L^2(\Omega)$-norm is robust with respect to the cut
configuration. Furthermore, a peculiar effect is observed as the
errors don't simply increase for decreasing $\varepsilon$, but peak at
certain values. This occurs because a smaller value of $\varepsilon$
doesn't simply imply thinner cut elements for this test case. The zones of the domain that constitute the thin (or sliver) cut
elements -- the parts of the domain for which $|x|>1$ or
$|y|>1$ -- do not only get thinner but also get shorter for
decreasing $\varepsilon$. The intersections between the domain
boundary and the grid lines $x = \pm 1$ and $y = \pm 1$ form the
boundaries of these zones, and move towards the points $x = \pm 1$,
$y = 0$ and $x = 0$, $y = \pm 1$ for $\varepsilon \rightarrow 0$. As a result, the cut elements become
thinner for decreasing $\varepsilon$ until these intersections shift
through a vertex, removing the thinnest cut element from the
system. At this point the error decreases, after which it again
increases until the intersection shifts through the next
vertex. Counting the number of degrees of freedom in the system
confirms that the errors peak at values of $\varepsilon$ where the
intersection between the boundary and the grid lines $x = \pm 1$ and
$y = \pm 1$ shift through a vertex, because at these values also the
number of degrees of freedom is reduced, which is indicated by the
vertical gray dash dotted lines in the figure.

\subsection{Higher order basis functions and three-dimensional domains}
The same effects as observed in the previous examples occur for higher order basis functions and three-dimensional domains. The fourth example has also been used to investigate the behavior of a second order discretization. The results are visible in Figure~\ref{fig:2ndOrder}. Note that this also requires a larger stabilization parameter as the space $V_h|_T^0$ in \eqref{eq:choice} is larger for higher order bases. In Figure~\ref{fig:threedimensional} the results of example 4 in a three-dimensional setting are visible. The domain is defined as $x^8+y^8+z^8 < (1+\varepsilon)^8$, Dirichlet conditions are imposed on the complete boundary and the approximated solution equals $u = \sin(\pi x) + \sin(\pi y) + \sin(\pi z)$. For this three-dimensional computation we have used a first order trilinear basis on a grid with size $h=1/8$ instead of $h=1/16$ and the maximal refinement depth has been reduced to 0 to save computation time. Also in three dimensions the same effect is observed.

\begin{figure}[h!]
  \begin{minipage}{.49\textwidth}
    \begin{center}
      \includegraphics[trim=4mm 4mm 4mm 2mm,width=\textwidth]{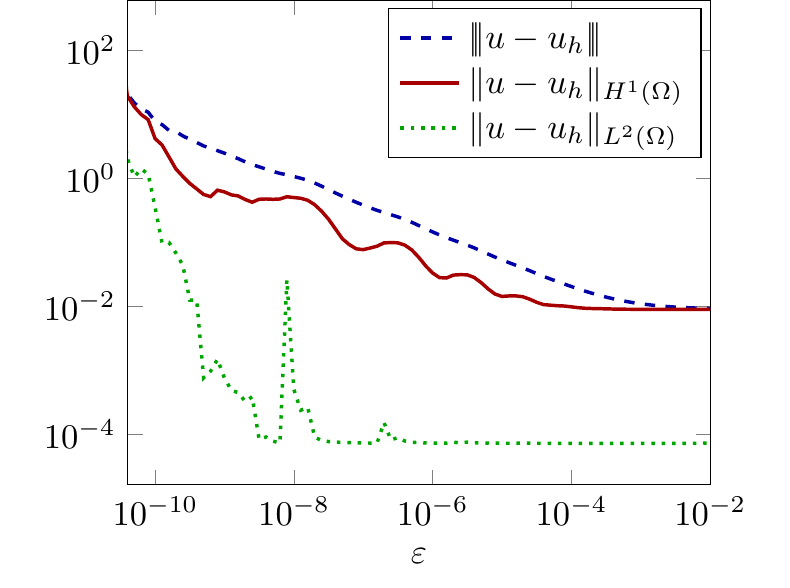}
      \caption{Second order B-splines.}\label{fig:2ndOrder}
    \end{center}
  \end{minipage}
  \begin{minipage}{.49\textwidth}
    \begin{center}
      \includegraphics[trim=4mm 4mm 4mm 2mm,width=\textwidth]{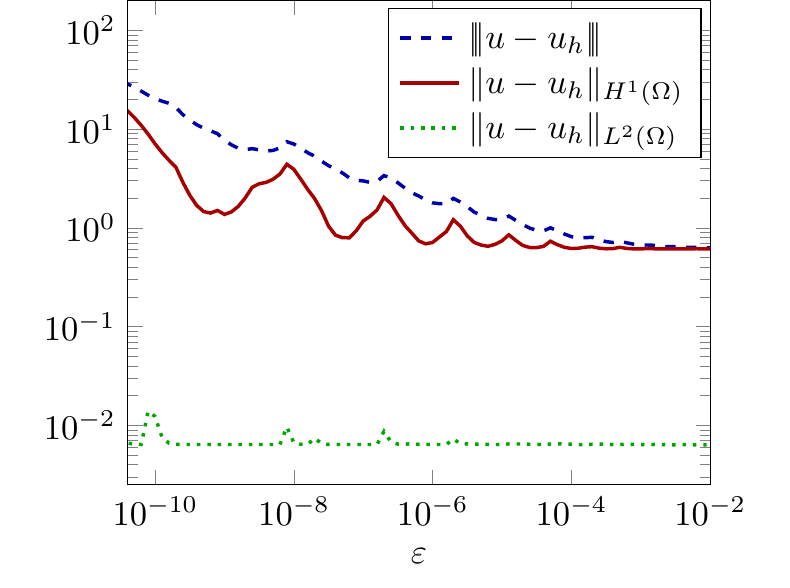}
      \caption{Three-dimensional example.}\label{fig:threedimensional}
        \end{center}
  \end{minipage}
\end{figure}

\subsection{Summary of observations}
In the previous examples we observed that sliver-like cut configurations
(on Dirichlet boundaries) can cause severe problems in the
discretization. The discretization error measured in the $H^1(\Omega)$-norm and
the mesh-dependent energy norm can not be bounded
independently of $\varepsilon$, the parameter that controls the
volume ratios in the sliver cut elements. While the $L^2(\Omega)$-norm error
seems to be robust for the simple cases, it is not for the more
complex examples.

We tried to give an interpretation of the origins that trigger
the large errors for small $\varepsilon$ and were able to
characterize a difference in the relative number of \dofs\ involved in the projection of the Dirichlet conditions by the
penalty terms. As a result, in Example 1 the effect is more severe for the
triangular discretization than for the quadrilateral discretization. On specific
domains, the instabilities due to the sliver cases are even
not triggered at all, as demonstrated in 
Example 2. This is similar to penalty methods for boundary
fitting discretizations, where exactly the right number of \dofs\ is involved in the penalty, such that very large penalties can be
applied \cite {babuska1973penalty}. However, small deviations
immediately display the missing robustness with respect to
$\varepsilon$ and lead to large errors, \emph{cf.}\ Example 3.
We have obtained similar results for a more complicated domain
with curved boundaries, for quadratic basis
functions and in three dimensions, such that the
results do not only apply for piecewise (bi-)linear discretizations
and two-dimensional problems on simple domains.

\section{Conclusion and resolutions of the problem}\label{sec:alternatives}
From Section~\ref{sec:analysis} and \ref{sec:numex} it is evident that the stabilization parameter needs to be bounded in order to guarantee optimal approximation properties. In fact, several modifications of Nitsche's method that preclude the previously discussed problems already exist in the literature. We divide these in three groups. Firstly, we discuss
alternatives to the symmetric Nitsche formulation. Secondly, there
exist approaches to preclude degrees of freedom that
have very small supports in the problem domain, such that the required
stabilization parameter is bounded. Thirdly, we discuss a
stabilization mechanism called the \emph{ghost penalty}, which achieves a
similar effect by penalizing jumps in normal derivatives along boundaries of cut elements.

\subsection{Alternative formulations}
The most obvious alternative for the symmetric Nitsche formulation is
the unsymmetric Nitsche formulation, \emph{e.g.,}
\cite{Burman2012,boiveau2016,Schillinger2016}, which however is not adjoint
consistent. Other methods to weakly impose Dirichlet conditions are
the penalty method and the use of Lagrange multipliers and related
methods, \emph{e.g.,}
\cite{babuska1973penalty,babuska1973lagrange,Fernandez2004,baiges2012}. Another
approach with unfitted grids is to manipulate the function space such
that it allows for the strong imposition of boundary conditions,
\emph{i.e.,} \cite{Hoellig2001,Hoellig2005,Sanches2011}.

\subsection{Precluding \dofs\ with small supports in the problem domain}
When the system does not contain functions with very small supports in
the problem domain, it follows from \eqref{eq:choice} that the 
required stabilization parameter is reduced. This is an additional
advantage of methods that eliminate functions with very small supports
from the system, which were initially introduced in order to improve
the conditioning of the system matrix. The simplest way to achieve
this is by simply excluding these basis functions from the basis,
\emph{e.g.,} \cite{reusken2008analysis,Embar2010,Sanches2011,Verhoosel2015}, which however
may decrease the accuracy. In
\cite{Hoellig2001,Hoellig2005,Rueberg2012,Rueberg2014,Rueberg2016},
functions with small supports in the problem domain are constrained to
geometrically nearby functions, which decreases the required
stabilization parameter without compromising the accuracy. Another
approach that precludes basis functions from having a very small \emph{stiffness} (having the same effect as precluding functions with a very small support) is
using a fictitious domain stiffness, \emph{e.g.,}
\cite{Schillinger2015}, which is thoroughly analyzed in
\cite{Dauge2014}.

\subsection{Ghost penalty}
The ghost penalty, see \cite{Burman2010,burman2012fictitious}, is a consistent penalty term that can be added to the bilinear form and which is customary in methods referred to as CutFEM, see \emph{e.g.,} \cite{burman2014cutfem}.
The ghost penalty introduces a penalty in the vicinity of the boundary that weakly bounds polynomials of neighboring elements to coincide. This is done by adding terms that penalize jumps in the normal derivatives along element interfaces or local projection type penalties \cite{becker01}. This weakly enforces higher order continuity along element interfaces near the boundary, and therefore relates functions with small support in the boundary region to interior functions. Similar to strongly enforcing this higher order continuity by constraining basis functions with small supports to geometrically nearby functions, this effectively controls and bounds the stabilization parameter. Besides bounding the stabilization parameter, penalizing the jumps in the normal derivatives gives control over the gradients in cut elements and therefore impede gradients of magnitude $\sim \varepsilon^{-1}$, which we observed in the cut elements.

\begin{remark}
Inspired by the observations in the experiments, we also suggest a quick fix that is simple to implement. To alleviate the problem of unbounded stabilization parameters, we modify Nitsche's method to a hybrid Nitsche-penalty method. First, we manually set an upper bound for the stabilization parameter $C_\lambda>0$. We then use Nitsche's method on the elements where the required stabilization parameter (computed as in \eqref{eq:choice}) is smaller than this upper bound and use a penalty method (\emph{i.e.,} no flux terms) with penalty parameter $C_\lambda$ on the elements where Nitsche's method would require a larger value than this manually set upper bound. Bounding the penalty parameter is easily motivated from an accuracy point of view, when considering the results and analysis previously presented in this manuscript. It is noted that a pure penalty method is not consistent with the initial problem \eqref{eq:evprob}. This inconsistency can be argued to be acceptable however, considering that this generally only occurs on very small parts of the domain and that on these parts of the boundary the flux terms are negligibly small compared to the penalty terms anyway.

This hybrid Nitsche-penalty method yields good approximations in the $H^1(\Omega)$-norm for the fourth test case, with $C_\lambda = 16 \cdot 10^3$. The results are visible in Figure~\ref{fig:ResultsNewModification} and show that the quality of the approximation is virtually independent of $\varepsilon$. This suggests that this could be a simple and straightforward fix for this problem. We consider this a preliminary result however, as this adaptation of the method has not undergone rigorous testing. Also setting the value of $C_\lambda$ needs a more thorough analysis and could require different values for different mesh sizes, higher order discretizations and other problems than Poisson's problem.  

\begin{figure}[h!]
  \begin{center}
    \includegraphics[trim=4mm 4mm 4mm 2mm,
    width=0.9\textwidth]{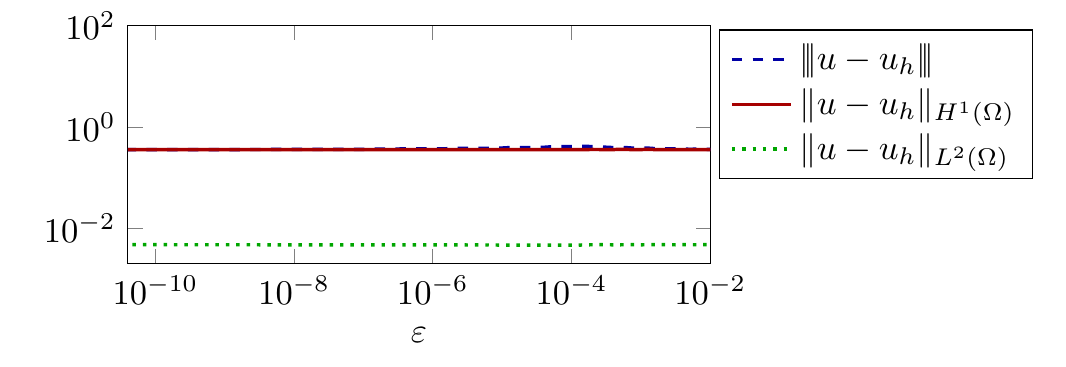}
  \end{center}
  \caption{Discretization errors of the modification of Nitche's method for Example 4. The energy norm error for this hybrid method is defined as: $\enorm{v}^2 := \Vert \nabla v \Vert_{\Omega}^2 + \Vert \lambda_T^{-\frac12} \partial_n v \Vert_{\Gamma^{\rm Nitsche}}^2 + \Vert \lambda_T^{\frac12} v \Vert_{\Gamma^{\rm Nitsche}}^2 + \Vert C_\lambda^{\frac12} v \Vert_{\Gamma^{\rm penalty}}^2$.}
  \label{fig:ResultsNewModification}
\end{figure}
\end{remark}

\section*{Acknowledgement}
\noindent
The research of F.\ de Prenter was funded by the NWO under the Graduate Program Fluid \& Solid Mechanics.
C. Lehrenfeld gratefully acknowledges funding by the German Science Foundation (DFG) within the project LE 3726/1-1.

\section*{References}
\bibliographystyle{elsarticle-num}
\bibliography{literature}

\end{document}